\documentclass[nohyperref]{article}

\usepackage{microtype}
\usepackage{graphicx}
\usepackage{subcaption}
\usepackage{booktabs} 

\usepackage{xcolor}         

\usepackage{hyperref}

\usepackage[accepted]{icml2022_preprint}

\usepackage{amsmath}
\usepackage{amssymb}
\usepackage{mathtools}
\usepackage{amsthm}

\usepackage[capitalize,noabbrev]{cleveref}

\theoremstyle{plain}
\newtheorem{theorem}{Theorem}[section]

\newtheorem{lemma}[theorem]{Lemma}

\theoremstyle{definition}
\newtheorem{definition}[theorem]{Definition}

\theoremstyle{remark}

\DeclareMathOperator{\rk}{rank}	
\DeclareMathOperator{\tr}{tr}	    
\DeclareMathOperator{\vecc}{vec}	    
\DeclareMathOperator{\mat}{mat}	    

\DeclarePairedDelimiter\ceil{\lceil}{\rceil}    
\DeclarePairedDelimiter\floor{\lfloor}{\rfloor} 
\DeclareMathOperator*{\argmin}{arg\,min}        

\newcommand{\RR}{\mathbb R}
\newcommand\mb{\mathbf}

\DeclarePairedDelimiter{\norm}{\lVert}{\rVert}

\usepackage[textsize=tiny]{todonotes}

\icmltitlerunning{A lifted framework for matrix sensing problems}

\begin{document}

\twocolumn[
\icmltitle{Over-parametrization via Lifting for Low-rank Matrix Sensing: \\ Conversion of Spurious Solutions to Strict Saddle Points }



\icmlsetsymbol{equal}{*}

\begin{icmlauthorlist}
\icmlauthor{Ziye Ma}{eecs}
\icmlauthor{Igor Molybog}{meta}
\icmlauthor{Javad Lavaei}{or}
\icmlauthor{Somayeh Sojoudi}{eecs}
\end{icmlauthorlist}

\icmlaffiliation{eecs}{Department of EECS, UC Berkeley, USA}
\icmlaffiliation{meta}{Meta AI}
\icmlaffiliation{or}{Department of IEOR, UC Berkeley, USA}

\icmlcorrespondingauthor{Ziye Ma}{ziyema@berkeley.edu}

\icmlkeywords{Machine Learning, ICML}

\vskip 0.3in
]



\printAffiliationsAndNotice{} 

\begin{abstract}


This paper studies the role of over-parametrization in solving non-convex optimization problems. The focus is on the important class of low-rank matrix sensing, where we propose an infinite hierarchy of non-convex problems via the lifting technique and the Burer-Monteiro factorization. This contrasts with the existing over-parametrization technique where the search rank is limited by the dimension of the matrix and it does not allow a rich over-parametrization of an arbitrary degree. We show that although the spurious solutions of the problem remain stationary points through the hierarchy, they will be transformed into strict saddle points (under some technical conditions) and can be escaped via local search methods. This is the first result in the literature showing that over-parametrization creates a negative curvature for escaping spurious solutions. We also derive a bound on how much over-parametrization is requited to enable the elimination of spurious solutions. 

\end{abstract}

\section{Introduction}
In this paper, we focus on an important class of non-convex optimization problems, named matrix sensing, which can be formulated as the feasibility problem:
\begin{align} \label{eq:SDP main}
    \mathrm{find}&\quad M\in\mathbb{R}^{n\times n} \\
    \mathrm{s.t.}& \quad \mathcal{A}(M) = \mathcal{A}(M^*) \notag\\
    & \quad \rk( M) \leq r, M \succeq 0.
    \notag
\end{align}
where the measurement operator $\mathcal{A}(\cdot): \mathbb{R}^{n \times n} \mapsto \mathbb{R}^{d}$ is a linear operator returning a $d$-dimensional measurement vector $\mathcal{A}(M)= [ \langle A_1, M \rangle, \dots, \langle A_d, M \rangle]^T$, for given sensing matrices $\{ A_i \}_{i=1}^d \in \mathbb{R}^{n \times n}$. The goal is to find an unknown matrix $M^*$ of rank $r$ associated with the measurement vector $b$, meaning that $b=\mathcal{A}(M^*)$, where $r$ is often much smaller than $n$. We factorize $M^*$ as $M^* = ZZ^\top$ where $Z \in \RR^{n \times r}$.

The problem \eqref{eq:SDP main} is extremely broad since solving any arbitrary polynomial optimization can converted to a series of problems in the form of \eqref{eq:SDP main} \cite{molybog2020conic}. In addition, the problem \eqref{eq:SDP main} directly arises in various applications such as collaborative filtering \citep{koren2009matrix}, phase retrieval \citep{singer2011angular,boumal2016nonconvex,shechtman2015phase}, motion detection \citep{fattahi2020exact}, and power system state estimation \citep{zhang2017conic,jin2019towards}. Moreover, its strikingly simple form is associated with only one source of non-convexity, which is the rank constraint. As a result, the existing works have extensively studied under what conditions one can recover $M^*$ exactly, and the centerpiece of this line of research is the notion of Restricted Isometry Property(RIP) of the measuring operator $\mathcal{A}$, which we state below:
\begin{definition}[RIP] \citep{candes2009exact} \label{def:rip}
    Given a natural number $p$, the linear map $\mathcal{A}: \mathbb{R}^{n \times n} \mapsto \mathbb{R}^{m}$ is said to satisfy $\delta_{p}$-RIP if there is a constant $\delta_{p} \in [0,1)$ such that
    \[ (1-\delta_{p})\|M \|_F^2 \leq \| \mathcal{A} (M) \|^2 \leq (1 + \delta_{p}) \|M\|_F^2 \]
    holds for all matrices $M \in \mathbb{R}^{n \times n}$ satisfying $\rk(M) \leq p$.
\end{definition}
This criterion is also intuitive to understand, as it measures how close $\mathcal{A}$ is to an identity operator (isometry) over matrices of rank at most $r$. If $\mathcal{A}$ is an exact isometry, or equivalently when it satisfies RIP with $\delta_{r} = 0$, we measure $M^*$ exactly and the recovery is trivial. Therefore, a small value for the RIP constant usually implies that the problem has a low computational complexity.

A popular approach to solving \eqref{eq:SDP main} is to use the so-called Burer-Monteiro (BM) factorization \cite{burer2003nonlinear}. The BM formulation explicitly factors $M$ as $M = XX^\top$, where $X \in \RR^{n \times r}$
\begin{equation}\label{eq:bm_main}
	\min_{X \in \mathbb{R}^{n \times r}} f(X) \coloneqq \frac{1}{2} \|\mathcal{A}(XX^T) - b\|^2.
\end{equation}
The problem \eqref{eq:bm_main} is an unconstrained smooth optimization problem, which means that highly scalable local search methods such as Gradient descent can be utilized to numerically solve it. Since the search is  over $\RR^{n \times r}$ instead of $\RR^{n \times n}$ in the original feasibility problem \eqref{eq:SDP main}, the number of variables is dramatically reduced from $\mathcal{O}(n^2)$ to $\mathcal{O}(nr)$, thereby improving its scalability. The main issue with \eqref{eq:bm_main} is that it is still a non-convex problem and thus it may contain spurious\footnote{A point is called spurious if it satisfies first-order and second-order necessary optimality conditions but is not a global minimum.} local minima, preventing local search methods from convergence to a global optimum. However, despite its non-convexity, a recent line of work \cite{zhang2019sharp,bi2020global,zhang2021general,ma2021sharp,ma2022noisy} has shown that if \eqref{eq:bm_main} satisfies the RIP condition with $\delta_{2r} < 1/2$, every local minimizer $\hat X$ of \eqref{eq:bm_main} will satisfy the relation $\hat X \hat X^\top = M^*$, precisely in the noiseless scenario and approximately when $b$ is corrupted by random noise. It has also been proven that $\delta_{2r} < 1/2$ is a sharp bound, meaning that there are counterexamples such that $\hat X \hat X^\top \neq M^*$ once $\delta_{2r} \geq 1/2$. This also falls in line with a prior result that $\delta_{2r} < 1/2$ is sufficient for recovering $M^*$ using specialized methods directly applied to \eqref{eq:SDP main} \cite{recht2010guaranteed,candes2011tight,cai2013sharp}.

\subsection{The power and limitation of over-parametrization}
The bound $\delta_{2r} < 1/2$ is sharp, and RIP conditions are difficult to satisfy and verify except for isometric Gaussian observations. In many applications, such as power system analysis, the RIP constant does not exist or is above $0.99$ \cite{zhang2019spurious}. Yet, it is highly desirable to transfer the scalability benefits of the BM factorization approach to these practical cases as well. Hence, it is essential to investigate how to handle problems that do not satisfy the RIP property with a constant smaller than $1/2$, using BM-type techniques.
Towards this end, an active line of research has studied the relationship between the complexity of recovering the global optimum and the degree of (over-) parametrization in \eqref{eq:bm_main} \cite{zhang2021sharp,zhang2022improved,levin2022effect}, and the results are promising.

The current idea of over-parametrization in matrix sensing consists of enlarging the search space of $X$ from $\RR^{n \times r}$ to $\RR^{n \times r_{\text{search}}}$, where $r_{\text{search}} \in [r,n)$, and we arrive at the following counterpart of \eqref{eq:bm_main}
\begin{equation}\label{eq:bm_overparam}
	\min_{X \in \mathbb{R}^{n \times r_{\text{search}}}} f(X) \coloneqq \frac{1}{2} \|\mathcal{A}(XX^T) - b\|^2.
\end{equation}
 The above-mentioned papers have shown that as $r_{\text{search}}$ increases, stronger guarantees for the recovery of $M^*$ can be obtained (although it requires stricter assumptions). One of the main results in this area will be stated below. 
\begin{theorem}[Theorem 1.1 of \citet{zhang2022improved}]\label{thm:richard_overparam}
	Assume that \eqref{eq:bm_overparam} satisfies the $(L_s,n)$-RSS (Restricted Strong Smoothness) and $(\alpha_s,n)$-RSC (Restricted Smooth Convexity) properties. If 
	\begin{equation} \label{eq:richard_ineq}
		r_{\text{search}} > \frac{1}{4} \left(\frac{L_s}{\alpha_s}-1 \right)^2 r, \quad r \leq r_{\text{search}} < n
	\end{equation}
	then every second-order point (SOP) $\hat X \in \RR^{n \times r_{\text{search}}}$ of \eqref{eq:bm_overparam} satisfies that $\hat X \hat X^\top = M^*$.
\end{theorem}
Note that $\hat X$ is an SOP if it satisfies the first-order and second-order necessary optimality conditions. The above theorem replaces the RIP condition with the similar conditions of RSS and RSC, which will be formally defined in the next section. The power of this theorem is in dealing with the scenario where $\delta_{2r} \geq 1/2$, by selecting a search rank $r_{\text{search}}>r$.

Despite the superiority of \eqref{eq:bm_overparam} over \eqref{eq:bm_main}, the power of the stated over-parametrization is limited. The reason is that $r_{\text{search}}$ cannot be greater than $n$ and therefore it is impossible to satisfy the condition \eqref{eq:richard_ineq} in practical cases where $L_s/\alpha_s$ is large. This calls for a new framework that accommodates an arbitrarily large degree of parametrization (as opposed to $r_{\text{search}} < n$), which would be effective in the regime of high $L_s/\alpha_s$ values. In this paper, we address this problem by proposing a tensor-based framework and analyzing its optimization landscape.

Our approach is related to over-parametrization used in the Semidefinite Programming (SDP) formulation. The SDP formulation is a natural convex relaxation of the original problem \eqref{eq:SDP main}, obtained by removing the rank constrained. It aims to minimize the nuclear norm of $M$ as a surrogate of its rank. When the search is performed on the space of symmetric and positive-semidefinite matrices, we can further reformulate the problem using a trace objective instead of the nuclear norm due to their equivalence under this setting. Hence, the resulting SDP formulation can be stated as
\begin{align} \label{eq:sdp_trace}
    \underset{{\bf M} \in \mathbb{R}^{n \times n}}{\mathrm{min}} &\quad \mathrm{tr}({\bf M}) \qquad
    \mathrm{s.t.} \ \ \mathcal{A}(\mathbf{M}) = b, \  \mathbf{M} \succeq 0, 
\end{align}
The relation $\delta_{2r} <1/2$ is a sufficient condition for the recovery of $M^*$ via \eqref{eq:sdp_trace}, but not a necessary one \cite{cai2013sharp}. Recently, \citet{yalcin2022semidefinite} showed that a sufficient bound close to $\delta_{2r} \leq 1$ can be achieved when $n \approx 2r$, which proves that the formulation \eqref{eq:sdp_trace} may solve the problem even if the sharp bound $\delta_{2r} \geq 1/2$ for \eqref{eq:bm_main} is not satisfied as long as $r$ is a large number (note that the focus of this paper is on the practical scenario of a small rank $r$). 

Overall, over-parametrization is a powerful idea since the inclusion of extra variables reshapes the landscape of the problem. Outside the realm of matrix sensing, the idea of constructing an infinite hierarchy of non-convex problems of increasing dimensions has been applied to the Tensor PCA problem \cite{wein2019kikuchi}. The empirical evidence of deep learning practice shows the advantage of using overparametrized models for both convergence properties during training \cite{oymak2020toward, zou2020gradient, du2019gradient, allen2019convergence} and generalization performance of the trained model \cite{allen2019learning, neyshabur2018the, mei2022generalization, belkin2020two}. Practitioners also design their own hierarchy of machine learning models, to satisfy the scaling laws \cite{kaplan2020scaling, hoffmann2022training, maloney2022solvable}. The cornerstone idea on the theoretical side of this field comes from the development of a hierarchy of convex problems, called the Sum-of-Squares hierarchy.

\subsection{Sum-of-Squares Optimization}

One of the most prominent over-parametrization frameworks for polynomial optimization is the framework of Sum-of-Squares (SOS) hierarchy of optimization problems \cite{parrilo2003semidefinite,lasserre2001global}. SOS optimization is essentially an optimization framework that leverages deep results in algebraic geometry to construct a hierarchy of convex problems of increasing qualities, solving each of which obtains a lower-bound certificate on the minimum value of the polynomial optimization problem of interest. Since \eqref{eq:bm_main} is also a polynomial optimization problem, SOS can be applied to handle the problem through a highly parametrized setting. Moreover, instead of using the usual SOS framework that finds a sequence of lower bounds on the optimal value of \eqref{eq:bm_main}, we could use its dual problem, since the minimum value of \eqref{eq:bm_main} is $0$ by construction. To construct the dual SOS problem, define $\kappa \geq 1$ to be an integer such that $2\kappa$ is equal to or larger than the maximum degree of $f(v)$ in \eqref{eq:bm_main}, where $v \coloneqq \vecc(X)$. Furthermore, define $[v]_\kappa \in \RR^s$ to be a vector containing the standard monomials of $v$ up to degree $\kappa$, with $s \coloneqq \binom{n+\kappa}{\kappa}$. We then build the moment matrix $D \coloneqq [v]_\kappa [v]_\kappa^\top$ with its entries being all standard monomials up to degree $2\kappa$. As a result, it is possible to rewrite $f(v)$ (i.e., $f(X)$) as a linear function of $D$, namely
\[
	f(v) = \langle F, D \rangle
\]
for some constant matrix $F \in \RR^{s \times s}$. Therefore, optimizing $\langle F, D \rangle$ is equivalent to optimizing $f(v)$ given that $D$ is rank-1 and positive-semidefinite. However, the rank-1 constraint is non-convex and its elimination leads to the dual SOS problem with the following form:
\begin{equation}\label{eq:sos_main}
	\min_{D \in \mathbb{S}^s}  \langle F, D \rangle \ \text{s.t.} \ \mathcal{L}(D) = 0, D \succeq 0
\end{equation}
The linear operator $\mathcal{L}$ captures the so-called consistency constraints, as some entries in $D$ may be identical due to being the outer product of monomial vectors. For example, if $n=2, \kappa=2$, we have
\[
	[v]_\kappa = [1, v_1, v_2, v_1^2, v_1 v_2, v_2^2 ]^\top
\]
meaning that $D_{15} = D_{23} = v_1v_2$, $D_{14} = D_{22} = v_1^2$, $D_{34} = D_{25} = v_1^2 v_2$, $D_{26} = D_{35} = v_1v_2^2$, and so on. The dual SOS problem \eqref{eq:sos_main} has some nice properties: it is convex and its optimal value asymptotically reaches that of \eqref{eq:bm_main} as $\kappa$ grows to infinity (under generic conditions), which enables solving the non-convex \eqref{eq:bm_main} with an arbitrary accuracy \cite{lasserre2001global}. However, the problem \eqref{eq:sos_main} also presents daunting challenges. 

First, it has poor scalability properties because it requires solving costly SDP problems. The idea behind this paper is related to applying the BM factorization to \eqref{eq:bm_main} (without dealing with SDPs) via a lifting technique similar to \eqref{eq:sos_main}. Currently, there is no guarantee that local minimizers of the BM formulation will translate to the minimizer of the convex problem \eqref{eq:sos_main}. The state-of-the-art result regarding the BM factorization states that this correspondence can be established only when $r(r+1)/2 \geq m$, where $m$ is the number of linear constraints \cite{boumal2016nonconvex}. In matrix sensing, since $r$ is small and $m$ is large, this result cannot be applied. 

Second, it is difficult to gauge how large $\kappa$ needs to be in order for the convex relaxation to be exact, meaning that one may need to use significant computational resources to solve an instance of \eqref{eq:sos_main} corresponding to some value of $\kappa$, only to discover that its solution does not provide useful information about the optimal solution of the original problem, promoting to repeat the process for a larger value of $\kappa.$ This also prevents the practical application of SOS as it is common to miscalculate in advance how computationally challenging it can be to solve \eqref{eq:bm_main} via the SOS framework.

\subsection{Our Approach}

In this paper, we build upon some of the core ideas of SOS optimization in order to construct a new framework for over-parametrization that addresses the current issues with \eqref{eq:sos_main}. The key observation is that $[v]_{\kappa}$ is highly similar to a symmetric rank-1 tensor, namely
\[
	[v]_{\kappa} \approx v^{\otimes \kappa} \in \RR^{n \circ \kappa}
\]
with the only difference being that $v^{\otimes \kappa}$ contains some terms appearing more than once, which implies that \eqref{eq:sos_main} could also be casted as an SDP based on the outer product of $v^{\otimes \kappa}$ with itself. The notion of a tensor and its properties (symmetry, rank, etc) will be formally introduced in Section \ref{sec:def} and Appendix A. Instead of solving a non-scalable SDP problem for the optimal $D$ over $\mathbb S^s$; we propose to apply local search over $\RR^{n \circ \kappa}$ for $v^{\otimes \kappa}$, and will analyze when it converges to the global optimum.

We first focus our attention to the $r=1$ case, which is easier to conceptualize. The idea is that we replace $\RR^n$ with $\RR^{n \times \dots \times n}$, meaning that we replace our decision variable with an $l$-way tensor $\mb{x} \in \RR^{n \times \dots \times n}$. We will show that this new problem can convert spurious solutions of the original problem to strict saddle points in the lifted tensor space, and we further derive how large $l$ should be in order for this to occur, thereby addressing two main practical deficiencies of \eqref{eq:sos_main}.

\section{Definitions and Notations}
\label{sec:def}
The formal definition of a tensor, alongside with the property of symmetry and rank will be elaborated in Appendix A.

\begin{definition}[Tensor Multiplication]
	Outer product is an operation carried out on a couple of tensors, denoted as $\otimes$. The outer product of 2 tensors $\mathbf{a}$ and $\mathbf{b}$, respectively of orders $l$ and $p$, is a tensor of order $l+p$, $\mb{c} = \mb{a} \otimes \mb{b}$ with
	\[
		c_{i_1\dots i_l j_1 \dots j_p} = a_{i_1\dots i_l} b_{j_1 \dots j_p}
	\]
	When the 2 tensors are of the same dimension, this product is such that $\otimes: \RR^{n \circ l} \times \RR^{n \circ p} \mapsto \RR^{n \circ (l+p)}$. Note we often use the shorthand
	\[
		 \underbrace{a \otimes \dots \otimes a}_{l \ \text{times}}\coloneqq  a^{\otimes l}
	\]
	We also define an inner product of two tensors. The mode-$q$ inner product between the 2 aforementioned tensors having the same $q$-th dimension is denoted as $\langle \mb{a}, \mb{b} \rangle_q$. Without loss of generality, assume $q = 1$ and 
	\[
		\left[ \langle \mb{a}, \mb{b} \rangle_q \right]_{i_2\dots i_l j_2 \dots j_p} = \sum_{\alpha=1}^{n_q} a_{\alpha i_2\dots i_l} b _{\alpha j_2 \dots j_p}
	\]
	Note that when we write $\langle \cdot, \cdot \rangle_q$, we count the $q$-th dimension of the first entry. Indeed, this definition of inner product can also be trivially extended to multi-mode inner products by just summing over all modes, denoted as $\langle \mb{a}, \mb{b} \rangle_{q, \dots, s}$.
\end{definition}

\begin{definition}[restricted strong smoothness (RSS)]
	The linear operator $\mathcal{A}: \RR^{n \times n} \mapsto \RR^m$ satisfies the $(L_s,r)$-RSS, property if:
	\[
		f(M) - f(N) \leq \langle M - N, \nabla f(M) \rangle + \frac{L_s}{2} \|M-N\|^2_F
	\]
	for all $M,N \in \mathbb{S}^n$ with $\rk(M), \rk(N) \leq r$.
\end{definition}

\begin{definition}[restricted strong convexity (RSC)]
	The linear operator $\mathcal{A}: \RR^{n \times n} \mapsto \RR^m$ satisfies the $(\alpha_s,r)$-RSC property if:
	\[
		f(M) - f(N) \geq \langle M - N, \nabla f(M) \rangle + \frac{\alpha_s}{2} \|M-N\|^2_F
	\]
	for all $M,N \in \mathbb{S}^n$ with $\rk(M), \rk(N) \leq r$.
\end{definition}


\subsection{Notations}

In this paper, $I_n$ refers to the identity matrix of size $n \times n$. The notation $M \succeq 0$ means that $M$ is a symmetric and positive semidefinite (PSD) matrix. $\mathbb{S}^n$ denotes the symmetric PSD space of dimension n. $\sigma_i(M)$ denotes the $i$-th largest singular value of a matrix $M$, and $\lambda_i(M)$ denotes the $i$-th largest eigenvalue of $M$. $\norm{v}$ denotes the Euclidean norm of a vector $v$, while $\norm{M}_F$ and $\norm{M}_2$ denote the Frobenius norm and induced $l_2$ norm of a matrix $M$, respectively. $\langle A,B \rangle$ is defined to be $\tr(A^TB)$ for two matrices $A$ and $B$ of the same size. For a matrix $M$, $\vecc(M)$ is the usual vectorization operation by stacking the columns of the matrix $M$ into a vector. For a vector $v \in \RR^{n^2}$, $\mat(v)$ converts $v$ to a square matrix and $\mat_S(v)$ converts $v$ to a symmetric matrix, i.e., $\mat(v)=M$ and $\mat_S(v) = (M+M^T)/2$, where $M \in \RR^{n \times n}$ is the unique matrix satisfying $v=\vecc(M)$. $[n]$ denotes the integer set $[1,\dots,n]$, and $\circ l$ stands for the shorthand of repeated cartesian product $\times \dots \times$ for $l$ times. $\mathcal N(\mu,\mathbf\Sigma)$ refers to the multivariate Gaussian distribution with mean $\mu$ and covariance $\mathbf\Sigma$.

\section{The lifted formulation}

To streamline the presentation, we focus on the problem of rank-1 matrix sensing  presented in the BM formulation:

\begin{equation}\label{eq:unlifted_main}
	\min_{x \in \RR^n} \quad \|\mathcal{A}(xx^\top - zz^\top)\|^2_2
\end{equation}
where $M^* = zz^\top$ is the ground truth rank-1 matrix. The generalization of the ideas to $r>1$ is straightforward but the mathematical notations will be cumbersome. 


The objective is to solve \eqref{eq:unlifted_main} using a lifted or over-parametrized framework. This means that instead of optimizing over the original vector space $\RR^n$, the goal is to optimize over a tensor space, namely $\RR^{n \circ l}$ for some $l \geq 2$. Note that \eqref{eq:unlifted_main} aims to find a vector $x$ such that
\[
	\mathcal{A}(xx^\top) = \mathcal{A}(M^*) = \mathcal{A}(zz^\top) \coloneqq b.
\]
Therefore, it is also desirable to achieve
\[
	\{\mathcal{A}(xx^\top)\}^{\otimes l} = b^{\otimes l} \in \RR^{m \circ l}
\] 
Define $\mathbf{B} \coloneqq \{\mathcal{A}(xx^\top)\}^{\otimes l}$. According to \cite{petersen2008matrix}, for arbitrary 4 vectors $a,b,c,d$ of the same dimension it holds that

\begin{equation}\label{eq:kron_iden}
	\langle a \otimes b, c\otimes d \rangle = \langle a , c \rangle \langle b, d\rangle
\end{equation}

With the repeated application of the above identity, we have that
\begin{equation}
	\begin{aligned}
		&B_{m_1\dots m_l} = \langle \prod_{k=1}^l \otimes \vecc(A_{m_k}), x^{\otimes l} \otimes x^{\otimes l} \rangle \\
		&= \sum_{i_1,\dots,i_l, j_1, \dots j_l}^n A_{m_1i_1j_1}\cdots A_{m_li_lj_l}(x_{i_1}\cdots x_{i_l}x_{j_1}\cdots x_{j_l})
	\end{aligned}
\end{equation}
Therefore, by defining the tensor $\mb{A} \in \RR^{(m \circ l) \times (n^2 \circ l)}$ as:
\begin{equation}\label{eq:lifted_a_def}
	A_{m_1\dots m_l} =  \prod_{k=1}^l \otimes \vecc(A_{m_k}),
\end{equation}
one can write the lifted objective similarly to \eqref{eq:unlifted_main}:
\begin{equation}\label{eq:lifted_main}
	\min_{\mb{w} \in \RR^{n \circ l}} \quad \| \langle \mb{A}, \mb{w} \otimes \mb{w} - z^{\otimes l} \otimes z^{\otimes l}\rangle_{m^l+1,\dots,m^l+n^{2l}} \|^2_F
\end{equation}

For notational convenience, define $f(\cdot): \RR^{n \times n} \mapsto \RR$ and $h(\cdot): \RR^{n} \mapsto \RR$ as:
\[
	f(M) \coloneqq \|\mathcal{A}(M - zz^\top)\|^2_2, \qquad h(x) = f(xx^\top),
\]
and $\nabla f(\cdot) = \nabla_M f(\cdot)$ and $\nabla h(\cdot) = \nabla_x h(\cdot)$. 

Similarly, define $f^l(\cdot): \RR^{n \circ 2l} \mapsto \RR$ and $h^l(\cdot): \RR^{n \circ l} \mapsto \RR$ as:
\begin{equation}
	\begin{aligned}
		&f^l(\mb{M}) \coloneqq  \| \langle \mb{A}, \mb{M} - z^{\otimes l} \otimes z^{\otimes l}\rangle_{m^l+1,\dots,m^l+n^{2l}} \|^2_F, \\ &h^l(\mb{w}) = f^l(\mb{w} \otimes \mb{w}),
	\end{aligned}
\end{equation}
as well as $\nabla f^l(\cdot) = \nabla_\mb{M} f^l(\cdot)$ and $\nabla h^l(\cdot) = \nabla_\mb{w} h^l(\cdot)$.

%
\begin{table*}[ht]
\centering
\renewcommand*{\arraystretch}{1}
\caption{The smallest eigenvalue of the Hessian of lifted SOPs of \eqref{eq:unlifted_main} }
\begin{tabular}{llllll}
\toprule
$n$ & $l$ & $\|\nabla h^l(z^{\otimes l})\|_F$ & $\|\nabla h^l(\hat x^{\otimes l})\|_F$ &$\lambda_{\min}(\nabla^2 h^l(z^{\otimes l}))$ & $\lambda_{\min}(\nabla^2 h^l(\hat x^{\otimes l}))$ \\
\midrule
3 & 1 & 0 & 0 & 3.99 & 2.67\\
3 & 2 & 0 & 0.003 & 3.99 & 0.61\\
3 & 3 & 0.004 & 0.002 & 3.99 & 0.24\\
3 & 4 & 0.006 & 0.004 & 3.99 & -0.17\\
5 & 1 & 0 & 0 & 4.18 & 1.87 \\
5 & 2 & 0.002 & 0 & 4.56 & -0.81 \\
7 & 1 & 0.002 & 0 & 4.35 & 1.89 \\
7 & 2 & 0.041 & 0 & 5.16 & -1.64 \\
\bottomrule
\end{tabular}
\label{table:Hessian_eigs}

\end{table*}

\section{A motivating example}

In this section, we study a class of benchmark matrix sensing instances that have many spurious local minima, where each instance $\mathcal{A}$ is defined as
\begin{equation} \label{eq:operator-ms}
    \mathcal{A}_{\epsilon}({\bf M})_{ij} := \begin{cases} {\bf M}_{ij}, & \text{if } (i,j) \in \Omega\\ \epsilon \mathbf{M}_{ij}, & \text{otherwise} \end{cases},
\end{equation}
where $\Omega$ is a measurement set such that
\[
	\Omega = \{ (i,i), (i,2k), (2k,i) | \ \ \forall i \in [n], k \in [ \floor{n/2}]\}
\]
\citet{yalcin2022semidefinite} has proved that each such instance has $\mathcal{O}(2^{\ceil{n/2}}-2)$ spurious local minima, while it satisfies the RIP property with $\delta_{2r} = (1-\epsilon)/(1+\epsilon)$ for some sufficiently small $\epsilon$.

To study whether our lifted framework can reshape the optimization landscape of the problem, we analyze the spurious local minima of the unlifted problem \eqref{eq:unlifted_main}. Given any spurious local minimum $\hat x$, it is essential to understand whether its lifted counterpart $\hat x^{\otimes l}$ behaves differently in \eqref{eq:lifted_main}, or more precisely whether $\hat x^{\otimes l}$ is still a spurious solution. To get some insight into this question, we conduct numerical experiments to first find the spurious solutions of \eqref{eq:unlifted_main} for the measurement matrices given in \eqref{eq:operator-ms}, and then find the smallest eigenvalue of the Hessian of \eqref{eq:lifted_main} at the lifted counterpart of each spurious solution. We summarize the findings in Table \ref{table:Hessian_eigs} for $\epsilon = 0.3$. Note that due to the structure of \eqref{eq:operator-ms}, the numbers of spurious local minimizers are equal for two cases $n$ and $n+1$ if $\ceil{n/2} = \ceil{(n+1)/2}$, and therefore the results for $n=4$ and $n=6$ are omitted.

It can be observed that, for a given spurious local minimizer $\hat x$ of \eqref{eq:unlifted_main}, two properties hold: (i) $\hat x^{\otimes l}$ is still a critical point as the gradient of the corresponding objective function $h^l$ is small (its nonzero value is due to the early stopping of the numerical algorithm), (ii) the Hessian at this point becomes smaller as $l$ increases. This means that as the degree of over-parametrization increases, the unlifted spurious local minima will become less of a local minima and more of a strict saddle point. This can be seen for $n=3$, as every increase in the parametrization leads to a reduced smallest eigenvalue and finally, $\hat x^{\otimes l}$ becomes a saddle point with a negative eigenvalue at level $l=4$, meaning that there is a viable escape direction for gradient descent algorithms. This trend can also be clearly observed for $n=5$ and $n=7$, implying that the transformation of the geometric properties at $\hat x^{\otimes l}$ is not an isolated phenomenon. To further study how much parametrization is needed and to show that this is not unique to any particular problem form, we provide theoretical results next.


\section{Optimization Landscape of the Lifted Problem}
\label{sec:landscape}


We analyze the optimization landscape of \eqref{eq:lifted_main} around $\hat x^{\otimes l}$, where $\hat x$ is a spurious spurious solution of \eqref{eq:unlifted_main}. The proofs to all of the results below can be found in Appendix B.

\subsection{FOP and SOP conditions}

\begin{lemma}\label{lem:cp_unlifted}
	The vector $\hat x$ is a SOP of \eqref{eq:unlifted_main} if and only if
	
\begin{subequations}
		\begin{gather}
		\nabla f(\hat x \hat x^\top)\hat x = 0, \label{eq:focp_unlifted} \\
		\begin{split}
		&2 \langle \nabla f(\hat x \hat x^\top ), u u^\top \rangle + \\
		[\nabla^2 f(\hat x \hat x^\top)](\hat x &u^\top + u \hat x^\top,\hat x u^\top + u \hat x^\top ) \geq 0		
		\end{split}\label{eq:socp_unlifted}
		\end{gather}
\end{subequations}

	$\forall u \in \RR^n $, with \eqref{eq:focp_unlifted} being the necessary and sufficient condition for $\hat x$ to be a first-order point (FOC), which is a stationary point of the objective function.
\end{lemma}

Lemma \ref{lem:cp_unlifted} has a counterpart in the lifted space since \eqref{eq:unlifted_main} and \eqref{eq:lifted_main} are highly similar. This will be formalized below. 
\begin{lemma}\label{lem:cp_lifted}
	The tensor $ \mb{\hat w}$ is a SOP of \eqref{eq:lifted_main} if and only if
	\begin{subequations}
		\begin{gather}
			\langle \nabla f^l( \mb{\hat w} \otimes  \mb{\hat w}) ,\mb{\hat w}\rangle_{n^l+1, \dots ,n^{2l}} = 0, \label{eq:focp_lifted}\\
		\begin{split}
			&2 \langle \nabla f(\mb{\hat w} \otimes  \mb{\hat w}), \Delta \Delta^\top \rangle +[\nabla^2 f(\mb{\hat w} \otimes  \mb{\hat w})] \\
			 (\mb{\hat w} \otimes \Delta + &\Delta \otimes \mb{\hat w},\mb{\hat w} \otimes \Delta + \Delta \otimes \mb{\hat w} ) \geq 0 \quad \forall \Delta \in \RR^{n \circ l}
		\end{split}
		\label{eq:socp_lifted}
		\end{gather}
	\end{subequations}
	with \eqref{eq:focp_lifted} being a necessary and sufficient condition for $\hat x$ to be a FOP.
\end{lemma}		

\subsection{Optimality condition of lifted problem, symmetric rank-1 constraint}	

\begin{theorem}\label{thm:focp_norank}
	For \eqref{eq:lifted_main}, the equality $\nabla h^l(\mb{\hat w}) = 0$ holds if
	\[
		\mb{\hat w} = \hat x^{\otimes l}
	\] 
	where $\hat x \in \RR^n$ is an FOP of \eqref{eq:unlifted_main}.
\end{theorem}


Theorem \ref{thm:focp_norank} theoretically confirms the phenomenon that we observed in the numerical example above, by asserting that all FOPs of the unlifted problem \eqref{eq:unlifted_main} are still FOPs in the lifted domain by transforming this point using tensor outer product, or by overparametrization. This means that some critical geometric structures of \eqref{eq:unlifted_main} are still maintained in \eqref{eq:lifted_a_def}, establishing strong connections between these two representations of the same problem.

After establishing the above negative result about a FOC having remained a stationary point after lifting, we turn to studying the differences between \eqref{eq:unlifted_main} and \eqref{eq:lifted_main}, since Table \ref{table:Hessian_eigs} suggests that the negative curvature of the Hessian at each spurious local minimizer will disappear after enough over-parametrization. Thus, the central problem under study is that whether for an instance of \eqref{eq:unlifted_main} with spurious local minima, these undesirable points will continue to be spurious solutions in the lifted space. If so, there is no apparent benefit to performing the lifting operation. Conversely, if the stationary points obtain a negative curvature, one can select from a large set of low-complexity algorithms to efficiently escape from strict saddles, and therefore eliminate local minima from the problem formulation.

The following theorem demonstrates that lifting enables the elimination of spurious solutions, and we can further derive a bound on the order $l$ needed to achieve the elimination of spurious solutions given the RSS and RSC constants of the problem.

\begin{theorem}\label{thm:socp_norank}
	Consider a SOP $\hat x \neq \pm z \in \RR^n$ of \eqref{eq:unlifted_main}, and assume that \eqref{eq:unlifted_main} satisfies the RSC and RSS conditions. Then $\mb{\hat w} = \hat x^{\otimes l}$ is a strict saddle of \eqref{eq:lifted_main} with a rank-1 symmetric escape direction if $\hat x$ satisfies the inequality
	\begin{equation}\label{eq:thm_socp_distance_lower}
		\|M^* - \hat x \hat x^\top \|^2_F \geq \frac{L_s}{\alpha_s} \|\hat x\|^2_2 \tr(M^*)
	\end{equation}
	and $l$ is odd and is large enough so that
	\begin{equation}\label{eq:thm_socp_lcond}
		l > \frac{1}{1-\log_2(2 \beta)} 
	\end{equation}
	where $\beta$ is defined as
	\[
		\beta \coloneqq \frac{ L_s \tr(M^*) \|\hat x\|^2_2}{\alpha_s \|M^* - \hat x \hat x^\top \|^2_F}.
	\]
	Here, $L_s,\alpha_s$ are the respective RSS and RSC constants of \eqref{eq:unlifted_main}.
\end{theorem}

Theorem \ref{thm:socp_norank} is powerful since it proves that by lifting spurious solutions to higher-order tensor spaces, we can convert them to saddle points, which attests to the power of over-parametrization. More importantly, regardless of how large $L_s/\alpha_s$ is, it is always possible to find an order $l$ large enough to convert $\hat x^{\otimes l}$s to saddle points, which is a major improvement over the existing results such as Theorem \ref{thm:richard_overparam}. Equation \eqref{eq:thm_socp_distance_lower} implies that in order for $\hat x^{\otimes l}$ to become a saddle point in the lifted formulation, it is enough to have either a small $\hat x$ or a large $\|M^* - \hat x \hat x^\top \|^2_F$. As a result, no spurious solution $\hat x$ close to the origin will remain spurious after lifting. This is a significant result because if we initialize a saddle-escaping algorithm near the origin, it will not be trapped inside a spurious solution during the early iterations even if the problem is highly non-convex. One fact to note is that \eqref{eq:thm_socp_distance_lower} is not a necessary condition, but a sufficient one, and therefore it is possible that the statements of Theorem \ref{thm:socp_norank} can be applied to a wider range of $\hat x$.

Another major advantage of Theorem \ref{thm:socp_norank} is that it quantifiably studies how many levels of parametrization are need in order to make an existing spurious local minimizer $\hat x$ a saddle point in the lifted space. Therefore, instead of offering a statement asserting that over-parametrization will work at some large enough $l$ (as done in the SOS setting), it explains how large this $l$ needs to be in terms of its geometric regularities, captured by the RSC and RSS parameters, and also the distance $\|M^* - \hat x \hat x^\top \|^2_F$.

Theorem \ref{thm:socp_norank} also implies that over-parametrization works particularly well for those spurious solutions $\hat x$ far away from the ground truth, in the sense that the further away $\hat x$ is from the ground truth, the smaller $l$ needs to be in order for $\hat x^{\otimes l}$ to become a saddle point, as suggested by \eqref{eq:thm_socp_lcond}. This fact is in line with the existing literature saying that the optimization landscape near $M^*$ is benign, in the sense that there exists no spurious local solution in a region around $M^*$. A well-known incarnation of the aforementioned statement is given below. 
\begin{theorem}[Theorem 3 \cite{bi2020global}]
	If $\hat x$ is a SOP of \eqref{eq:unlifted_main} and
	\begin{equation}\label{eq:thm_local_region_range}
		\|\hat x \hat x^\top - M^*\|_F \leq \frac{4L_s \alpha_s}{(L_s+\alpha_s)^2} \|z\|^2_2,
	\end{equation}
	then
	\[
		\hat x \hat x^\top = M^*
	\]
\end{theorem}
This means that any spurious solution must violate the inequality \eqref{eq:thm_local_region_range}. This allows us to simplify the results of Theorem \ref{thm:socp_norank}, which can be stated in the following form
%
\begin{theorem}\label{thm:socp_corollary}
	Assume that $\hat x \neq \pm z \in \RR^n$ is a spurious solution of \eqref{eq:unlifted_main}, and that \eqref{eq:unlifted_main} satisfies the RSC and RSS assumptions with $\alpha_s$ and $L_s$ constants respectively. The point $\hat x^{\otimes l}$ will become a saddle point of \eqref{eq:lifted_main} for an odd $l$ satisfying \eqref{eq:thm_socp_lcond} if
	\begin{equation}
		\|M^*\|_F \leq \frac{2\sqrt{2} \alpha_s^{5/2}}{(L_s+\alpha_s)^2 \sqrt{L_s}}
	\end{equation}
\end{theorem}
This theorem is proved by setting the RHS of \eqref{eq:thm_local_region_range} to be smaller than that of \eqref{eq:thm_socp_distance_lower}. Another technical lemma is introduced to bridge the two terms, so that it does not depend on specific $\hat x$ anymore.

The above results all aim to convert spurious solutions to saddle points via lifting. Although this property is highly desirable for spurious local solutions, it is essential to make sure that it will not hold for the ground truth solution since the correct solution should remain a SOP in the lifted space in order for the lifting technique to be useful. 

In our previous numerical experiment, we empirically showed that the smallest eigenvalue of the Hessian at $z^{\otimes l}$ remains positive, meaning that it is still a SOP in the lifted formulation \eqref{eq:lifted_main}. In the following theorem we formally establish this observation.

\begin{theorem}\label{thm:socp_z}
	Assume that $z \in \RR^n$ is the ground truth solution of \eqref{eq:unlifted_main}. Then $z^{\otimes l}$ remains a SOP of \eqref{eq:lifted_main} regardless of the parametrization level $l$, and without the need for \eqref{eq:unlifted_main} to satisfy RSC or RSS conditions.
\end{theorem}

\section{Numerical Experiments}

\begin{figure*}[t]
    \centering
    \begin{subfigure}{6cm}
    	    \includegraphics[width=\linewidth]{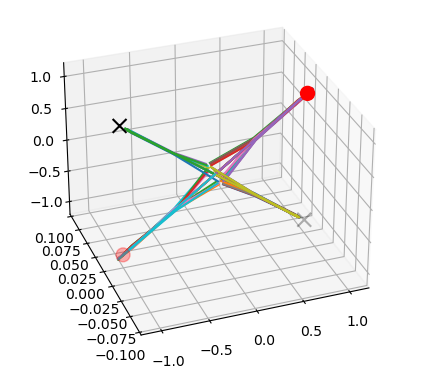}
    \caption{Convergence trajectories of unlifted formulation.}
    \end{subfigure} \hspace{2em}
    \begin{subfigure}{6cm}
    	    \includegraphics[width=\linewidth]{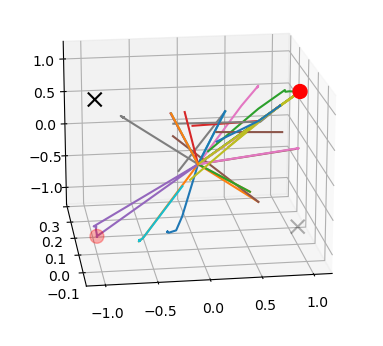}
    \caption{Convergence trajectories of lifted formulation with $l=3$.}
    \end{subfigure}
    \caption{The convergence trajectories of \eqref{eq:operator-ms}, with $n=3, \epsilon = 0.3$. Random gaussian initialization with $\sigma=0.01, \mu=0$. 40 Trials in total.}
    \label{fig:convergence}
\end{figure*}

In this section, we numerically demonstrate that the theoretical results of this paper can be translated to real advantages when using the lifted framework \eqref{eq:lifted_main}\footnote{The code used to produce results in this section can be found at: https://github.com/anonpapersbm/liftedmatrixsensing}.

For the sake of convenience, we revisit the matrix sensing problem \eqref{eq:unlifted_main} with $n=3$ and the special operator \eqref{eq:operator-ms}. We choose $\epsilon = 0.3$ in the numerical experiment, which translates to the RIP constant of $\delta_{2r} = 0.52$, going beyond the known sharp threshold of $\delta < 1/2$, and may create spurious solutions. By the special structure of \eqref{eq:operator-ms}, it is easy to verify that there are theoretically 4 SOPs in total, and they converge to the following 4 points as $\epsilon$ becomes sufficiently small, which are:
\[
	\hat x_1 \approx \begin{bmatrix} 1 \\ 0 \\1 \end{bmatrix}, \hat x_2 \approx \begin{bmatrix} 1 \\ 0 \\ -1 \end{bmatrix}, \hat x_3 \approx \begin{bmatrix} -1 \\ 0 \\ 1 \end{bmatrix}, \hat x_4 \approx \begin{bmatrix} -1 \\ 0 \\-1 \end{bmatrix}.
\]
in which $\hat x_1$ and $\hat x_4$ are ground truth solutions as $\hat x_1 \hat x_1^\top = \hat x_4 \hat x_4^\top = M^*$. The other SOPs $\hat x_2$ and $\hat x_3$ are spurious solutions. 

To empirically verify that $\epsilon = 0.3$ is indeed small enough, we simply start from random Gaussian initialization, and apply optimization algorithms to check to what point(s) the algorithm will eventually converge to. We use the standard ADAM optimizer \cite{kingma2014adam} with the hyper-parameter $\text{lr}=0.02$, and the 3D convergence trajectories are plotted in Figure \ref{fig:convergence}(a) for 40 different trials with independently sampled initial points. In this plot, the ground truth $\hat x_1$ and $\hat x_4$ are labeled with big red dots, and $\hat x_2$ and $\hat x_3$ are labeled with black crosses. One can easily observe that the theoretically derived SOPs are indeed correct, as the plot shows that regardless of initialization, the algorithm will always converge to one of the 4 points given above, which means that $\epsilon=0.3$ is already small enough to deteriorate the landscape. Upon a closer scrutiny, one can further realize that all 4 SOPs are equally attractive, and it is impossible to differentiate between ground truth solutions and spurious solutions. In particular, the success rate of applying ADAM to \eqref{eq:unlifted_main} with \eqref{eq:operator-ms} is 57.5\%. This is highly undesirable in practice because the user will constantly obtain different results by running the same algorithm, leading to confusion as to which result is correct, which exactly represents the inherent difficulty of a highly non-convex optimization problem like \eqref{eq:operator-ms}.

Thus, at a high level, it is necessary to show that by using the lifted framework \eqref{eq:lifted_main}, we can avoid converging to $\hat x_2^{\otimes l}$ and $\hat x_3^{\otimes l}$ since with this over-parametrized framework, it is possible that they have become saddle points instead of spurious solutions, as suggested by Theorem \ref{thm:socp_norank}. To this end, we plot the optimization trajectory of \eqref{eq:lifted_main} with $l=3$ and \eqref{eq:operator-ms} in Figure \ref{fig:convergence}(b), where the optimizer of choice is still ADAM, since it has the ability to escape saddle points and it makes the comparison with Figure \ref{fig:convergence}(a) meaningful. The reason that we chose $l=3$ instead of $l=2$ is because Theorem \ref{thm:socp_norank} only applies to odd values of $l$. However, one caveat is that since the optimization is performed in tensor space, it is impossible to visualize. To address this issue, instead of showing the full tensor, we perform tensor PCA along each step of the trajectory, and plot the 3D vector that can be transformed to the dominant rank-1 symmetric tensor via tensor outer product. In particular, given a tensor $\mb{w}$ on the trajectory, we plot $w \in \RR^3$ such that:
\[
	w = \argmin_w \| \mb{w} - w^{\otimes l} \|_F
\]
meaning that $w$ is the best projection of $\mb{w}$ onto $\RR^3$. This is why Figure \ref{fig:convergence}(b) seems more complicated than Figure \ref{fig:convergence}(a), as an extra layer of approximation is applied. Nevertheless, the message of Figure \ref{fig:convergence}(b) is unchanged, as now instead of converging to all 4 points equally, the lifted formulation only converges to the ground truth solutions, as no trajectory leads to the black crosses. This indicates that by converting $\hat x_2^{\otimes l}$ and $\hat x_3^{\otimes l}$ to saddle points via over-parametrization, we gain real benefits by avoiding spurious solutions, especially compared side-by-side with Figure \ref{fig:convergence}(a). To further demonstrate the power of the over-parametrized framework \eqref{eq:lifted_main}, we summarize the success rate of unlifted framework \eqref{eq:unlifted_main} and the lifted framework \eqref{eq:lifted_main} in the table below.
\begin{table}[ht]
\centering
\caption{Success rate of the lifted and unlifted frameworks applied to \eqref{eq:operator-ms}}
\begin{tabular}{lll}
\toprule
 & Unlifted $l=1$ & Lifted $l=3$ \\
\midrule
$n=3$ & 0.575 & 0.62 \\
$n=4$ & 0.575 & 0.68 \\
$n=5$ & 0.475 & 0.75 \\
\bottomrule
\end{tabular}
\label{table:success_rate}
\end{table}
Here, we count a trial to be a "success" if the final iteration $x_{\text{terminal}}$ satisfies 
\[
	\|x_{\text{terminal}}^{\otimes l} - z^{\otimes l} \|_F \leq 0.05
\]
From Table \ref{table:success_rate}, we can see that $n$ increases, the success rate of the lifted framework goes up, especially in contrast to the fact that higher $n$ means lower success rate for the unlifted formulation due to it having $\mathcal{O}(2^{\ceil{n/2}}-2)$ spurious local solution. This empirically demonstrates that the lifted formulation is especially valuable in problems with higher dimensions.

\section{Conclusion}
This paper proposed a powerful method to deal with the non-convexity of the matrix sensing problem via the popular BM formulation. Since the problem has several spurious solutions in general and local search methods are prone to be trapped in those points, we developed a new framework via a SOS-type lifting technique to address the issue. We show that although the spurious solutions remain stationary points through the lifting, if a sufficiently rich over-parametrization is used, those spurious solutions will be transformed into strict saddle points (under technical assumptions) and are escapable. This establishes the first result in the literature proving the conversion of spurious solutions to saddle points, and it quantifies how much over-parametrization is needed to break down the complexity of the problem. Future research directions include the sparsification of the lifting method to eliminate unnecessary monomials and reduce the complexity, as well as studying whether lifting will create new stationary points and where they are located relative to the ground truth solution.

\bibliography{references}
\bibliographystyle{icml2022}

\newpage
\appendix
\onecolumn
\section{Definition}
\begin{definition}[Tensor]
	As a generalization of the way vectors are used to parametrize finite-dimensional vector spaces, we use \emph{arrays} to parametrize tensors generated from product of finite-dimensional vector spaces, as per \cite{comon2008symmetric}. In particular, we define an $l$-way array as such:
	\[
		\mathbf{a} = \{a_{i_1i_2\dots i_l} | 1 \leq i_k \leq n_k, 1 \leq k \leq l \} \in \RR^{n_1 \times \dots \times n_l}
	\]
	Note that in this paper tensors and arrays can be regarded as synonymous since there exists an isomorphism between them. Moreover, if $n_1 = \dots = n_l$, then we call this tensor(array) an $l$-order(way), $n$-dimensional tensor. For the convenience of tensor representation, we use the notation $\RR^{n \circ l}$ with $n \circ l \coloneqq n \times \dots \times n$. In this work, tensors are denoted with bold variables, and other fonts are reserved for matrices, vectors, and scalars unless specified otherwise.
\end{definition}

\begin{definition}[Symmetric Tensor]
	Similar to the definition of symmetric matrices, for an order-$l$ tensor $\mb{a}$ with the same dimensions (i.e., $n_1 = \dots = n_l$), also called a cubic tensor, it is said that the tensor is symmetric if  its entries are invariance under any permutation of their indices:
	\[
		a_{i_{\sigma(1)} \cdots i_{\sigma(l)}} = a_{i_1 \cdots i_l} \quad \forall \sigma, \quad i_1, \dots, i_l \in \{1,\dots,n\}
	\]
	where $\sigma \in \mathcal{G}_l$ denotes a specific permutation and $\mathcal{G}_l$ is the symmetric group of permutations on $\{1, \dots, l\}$. We denote the set of symmetric tensors as $\mathrm{S}^l(\RR^n)$.
\end{definition}

\begin{definition}[Rank of Tensors]
	The rank of a cubic tensor $\mb{a} \in \RR^{n \circ l}$ is defined as:
	\[
		\rk(\mb{a}) = \min\{r | \mb{a} = \sum_{i=1}^r u_i \otimes v_i \otimes \cdots \otimes w_i \}
	\]
	where $u_i, \dots, w_i \in \RR^n \ \forall i$. Furthermore, according to \cite{kolda2015numerical}, if $\mb{a}$ is a symmetric tensor, then it can be decomposed as:
	\[
		\mb{a} = \sum_{i=1}^r \lambda_i u_i \otimes \dots \otimes u_i \coloneqq \sum_{i=1}^r \lambda_i u_i^{\otimes l}
	\]
	and the rank is conveniently defined as the number of nonnegative $\lambda_i$s, which is very similar to the rank of symmetric matrices indeed. For notational convenience, we denote rank-$r$ symmetric tensors as $\mathrm{S}^l(\RR^n)_r$.
	
\end{definition}

\section{Proofs}
\begin{proof}[Proof of Theorem \ref{thm:focp_norank}]
\eqref{eq:focp_unlifted} implies that
	\begin{equation}\label{eq:focp_unlifted_idx}
		\sum_{a,i,j,s}^{[m] \times [n] \times [n] \times [n]} (A_a)_{sk} (A_a)_{ij} \hat x_i \hat x_j \hat x_s = \sum_{a,i,j,s}^{[m] \times [n] \times [n] \times [n]} (A_a)_{sk} (A_a)_{ij} z_i z_j \hat x_s \qquad \forall k
	\end{equation}
	Then we focus on \eqref{eq:focp_lifted} with
	\[
		\nabla f^l( \mb{\hat w} \otimes  \mb{\hat w}) = \langle \mathbf{A}^\top \mathbf{A}, \mb{\hat w} \otimes \mb{\hat w} - z^{\otimes l} \otimes z^{\otimes l}\rangle_{n^{2l}+1,\dots,2n^{2l}}
	\]
	where $\mathbf{A}^\top \mathbf{A} \coloneqq \langle \mb{A}, \mb{A} \rangle_{m^l+1,\dots,m^l+n^{2l}}$.
	Thus, the LHS of \eqref{eq:focp_lifted} is:
	\begin{equation} \label{eq:lhs_focp_lift}
		\begin{aligned}
			&\sum_{\{a_\alpha, i_\alpha, j_\alpha, s_\alpha\}_{\alpha=1}^l}^{([m] \circ l) \times ([n] \circ l) \times ([n] \circ l) \times ([n] \circ l)} \left( \prod_{\alpha=0}^l A_{a_\alpha s_\alpha k_\alpha} A_{a_\alpha i_\alpha j_\alpha} \hat x_{i_\alpha} \hat x_{j_\alpha} \hat x_{s_\alpha} \right) - \left( \prod_{\alpha=0}^l A_{a_\alpha s_\alpha k_\alpha} A_{a_\alpha i_\alpha j_\alpha} z_{i_\alpha} z_{j_\alpha} \hat x_{s_\alpha} \right) \\
			= & \prod_{\alpha=0}^l \left( \sum_{a_\alpha, i_\alpha, j_\alpha, s_\alpha}^{[m] \times [n] \times [n] \times [n]} A_{a_\alpha s_\alpha k_\alpha} A_{a_\alpha i_\alpha j_\alpha} \hat x_{i_\alpha} \hat x_{j_\alpha} \hat x_{s_\alpha} \right) - 
			\prod_{\alpha=0}^l \left( \sum_{a_\alpha, i_\alpha, j_\alpha, s_\alpha}^{[m] \times [n] \times [n] \times [n]} A_{a_\alpha s_\alpha k_\alpha} A_{a_\alpha i_\alpha j_\alpha} z_{i_\alpha} z_{j_\alpha} \hat x_{s_\alpha} \right)
		\end{aligned}
	\end{equation}
	Given \eqref{eq:focp_unlifted_idx} and \eqref{eq:lifted_a_def}, we know that:
	\[
		\sum_{a_\alpha, i_\alpha, j_\alpha, s_\alpha}^{[m] \times [n] \times [n] \times [n]} A_{a_\alpha s_\alpha k_\alpha} A_{a_\alpha i_\alpha j_\alpha} \hat x_{i_\alpha} \hat x_{j_\alpha} \hat x_{s_\alpha} = 
		\sum_{a_\alpha, i_\alpha, j_\alpha, s_\alpha}^{[m] \times [n] \times [n] \times [n]} A_{a_\alpha s_\alpha k_\alpha} A_{a_\alpha i_\alpha j_\alpha} z_{i_\alpha} z_{j_\alpha} \hat x_{s_\alpha} \quad \forall k_\alpha 
	\]
	Therefore, substituting the above equality into \eqref{eq:lhs_focp_lift} yields that LHS of \eqref{eq:focp_lifted} is 0.
	
\end{proof}

Before proceeding to the proof of Theorem \ref{thm:socp_norank}, we first recall a useful technical Lemma from \cite{ma2022noisy}:
\begin{lemma}\label{lem:G_upper}
	For any SOP $\hat x$ of \eqref{eq:unlifted_main}, define $G$ as $G \coloneqq - \lambda_{\text{min}}(\nabla f(\hat x \hat x^\top))$, and $L_s$ be the RSS constant. Then it holds that:
	\[
		G \leq \|\hat x\|^2_2 L_s
	\]
\end{lemma}

\begin{proof}[Proof of Theorem \ref{thm:socp_norank}]
	\eqref{eq:socp_unlifted} implies that:
	\begin{equation}
	\begin{aligned}
		\text{LHS} =2 \sum_{a,i,j,s, k}^{[m] \times [n] \circ 4} (A_a)_{sk} (A_a)_{ij} (\hat x_i \hat x_j -z_i z_j) u_k u_s + \\
		\sum_{a,i,j,s,k}^{[m] \times [n] \circ 4} (A_a)_{sk} (A_a)_{ij} (\hat x_i u_j + u_i \hat x_j) (\hat x_s u_k + u_s \hat x_k)
	\end{aligned}
	\end{equation}
	If the $A_a$ matrices are symmetric, which can be achieved by redefining $A_a$ as $(A_a^T+A_a)/2$ without changing the measurement values, the above equation can be simplified as:
	\begin{equation}
	\begin{aligned}
		 2 \underbrace{\sum_{a,i,j,s}^{[m] \times [n] \circ 4} (A_a)_{sk} (A_a)_{ij} (\hat x_i \hat x_j -z_i z_j) u_k u_s}_{C_1} + 
		4 \underbrace{\sum_{a,i,j,s}^{[m] \times [n] \circ 4} (A_a)_{sk} (A_a)_{ij} \hat x_i \hat x_k u_j u_s}_{C_2}
	\end{aligned}
	\end{equation}
	According to \cite{zhang2021general}, $\nabla f(M)$ can be assumed to be symmetric without loss of generality. Hence, one can select $u \in \RR^n$ such that $u^\top \nabla f(\hat x \hat x^\top)u = \lambda_{\text{min}}(\nabla f(\hat x \hat x^\top))$ and $\lambda_{\text{min}}(\nabla f(\hat x \hat x^\top)) \leq 0$ under the RSC assumption with $\alpha_s \geq 0$. The reason that this holds is because first we know that
	\[
		f(M^*) \geq f(\hat x \hat x^\top) + \langle \nabla f(\hat x \hat x^\top ), M^* - \hat x \hat x^\top \rangle + \alpha_s \|\hat x \hat x^\top -M^*\|^2_F.
	\]
	Since $\langle \nabla f(\hat x \hat x^\top ), \hat x \hat x^\top \rangle = 0$ according to \eqref{eq:focp_unlifted} and $f(\hat x \hat x^\top) - f(M^*) \geq 0$, we know that
	\[
		\langle \nabla f(\hat x \hat x^\top ), M^* \rangle \leq -\alpha_s \|\hat x \hat x^\top -M^*\|^2_F
	\]
	after rearrangements. Furthermore, since both $\nabla f(\hat x \hat x^\top )$ and $M^*$ are assumed to be positive semidefinite for the above-mentioned reasons, we have that
	\[
		\langle \nabla f(\hat x \hat x^\top ), M^* \rangle \geq \lambda_{\min} (\nabla f(\hat x \hat x^\top )) \tr(M^*)
	\]
	which implies that 
	\begin{equation}\label{eq:G_bound}
		\lambda_{\min} (\nabla f(\hat x \hat x^\top )) \leq -\alpha_s \frac{\|\hat x \hat x^\top -M^*\|^2_F}{\tr(M^*)} \leq 0
	\end{equation}
	With this piece of knowledge in mind, we define $G \coloneqq - \lambda_{\text{min}}(\nabla f(\hat x \hat x^\top)) \geq 0$. Thus,
	\[
		C_1 = -G.
	\]
	Moreoever, the RSS condition implies that:
	\begin{equation*}
		\begin{aligned}
			4 C_2 &= [\nabla^2 f(\hat x \hat x^\top)](\hat x u^\top + u \hat x^\top , \hat x u^\top + u \hat x^\top) \leq L_s \|\hat x u^\top + u \hat x^\top\|^2_F  \\
			&= L_s \tr((\hat x u^\top + u \hat x^\top)^\top \hat x u^\top + u \hat x^\top) = 2 L_s \|\hat x\|^2_2
		\end{aligned}
	\end{equation*}
	since $u^\top \hat x = 0$ according to the first-order condition \eqref{eq:focp_unlifted}. Therefore,
	\[
		C_2 \leq \frac{1}{2} L_s \|\hat x\|^2_2
	\]
	
	Now, we take a look at the left-hand side (LHS) of \eqref{eq:socp_lifted}; here we choose $\Delta = u^{\otimes l}$ for the same $u \in \RR^n$ chosen above:
	\begin{equation}
		\begin{aligned}
		 &\left[ 2 \underbrace{\sum_{\{a_\alpha, i_\alpha, j_\alpha, s_\alpha, k_\alpha \}_{\alpha=1}^l}^{([m] \circ l) \times \left[ ([n] \circ l) \circ 4 \right]} \left( \prod_{\alpha=0}^l A_{a_\alpha s_\alpha k_\alpha} A_{a_\alpha i_\alpha j_\alpha} \hat x_{i_\alpha} \hat x_{j_\alpha} u_{s_\alpha} u_{k_\alpha}\right) - \left( \prod_{\alpha=0}^l A_{a_\alpha s_\alpha k_\alpha} A_{a_\alpha i_\alpha j_\alpha} z_{i_\alpha} z_{j_\alpha} u_{s_\alpha} u_{k_\alpha} \right)}_{C_3} \right] + \\
		&4 \underbrace{\sum_{\{a_\alpha, i_\alpha, j_\alpha, s_\alpha, k_\alpha \}_{\alpha=1}^l}^{([m] \circ l) \times \left[ ([n] \circ l) \circ 4 \right]} \prod_{\alpha=0}^l A_{a_\alpha s_\alpha k_\alpha} A_{a_\alpha i_\alpha j_\alpha} \hat x_{i_\alpha} \hat x_{k_\alpha} u_{j_\alpha} u_{s_\alpha}}_{C_4}
		\end{aligned}
	\end{equation}
	Now,
	\begin{equation}
		\begin{aligned}
			C_3 &= 2 \prod_{\alpha=1}^l \left( \sum_{a,i,j,s, k}^{[m] \times [n] \circ 4} A_{a_\alpha s_\alpha k_\alpha} A_{a_\alpha i_\alpha j_\alpha} \hat x_{i_\alpha} \hat x_{j_\alpha} u_{s_\alpha} u_{k_\alpha} \right) - 2 \prod_{\alpha=1}^l \left( \sum_{a,i,j,s, k}^{[m] \times [n] \circ 4} (A_{a_\alpha s_\alpha k_\alpha} A_{a_\alpha i_\alpha j_\alpha} z_{i_\alpha} z_{j_\alpha} u_{s_\alpha} u_{k_\alpha} \right) \\
			& = 2  \left( \sum_{a,i,j,s, k}^{[m] \times [n] \circ 4} (A_a)_{sk} (A_a)_{ij} \hat x_i \hat x_j  u_k u_s \right)^l - 2 \left( \sum_{a,i,j,s, k}^{[m] \times [n] \circ 4} (A_a)_{sk} (A_a)_{ij} z_i z_j  u_k u_s \right)^l \\
			&\leq 2 \left( \sum_{a,i,j,s, k}^{[m] \times [n] \circ 4} (A_a)_{sk} (A_a)_{ij} (\hat x_i \hat x_j -z_i z_j)  u_k u_s\right)^l = C_1^l = -  G^l
		\end{aligned}
	\end{equation}
	where the inequality follows from:
	\[
		a^n - b^n \leq (a-b)^n \quad \forall b \geq a \geq 0
	\]
	Here, since $a-b = C_1 \leq 0$, the above inequality can be used.
	Next,
	\begin{equation}
		\begin{aligned}
			C_4 = (\sum_{a,i,j,s}^{[m] \times [n] \circ 4} (A_a)_{sk} (A_a)_{ij} \hat x_i \hat x_k u_j u_s)^l = C_2^l  \leq \frac{1}{2^l} L^l_s \|\hat x\|^{2l}_2
		\end{aligned}
	\end{equation}
	As a result,
	\[
		\text{LHS of \eqref{eq:socp_lifted}} \leq \underbrace{-2G^l}_{\text{Part 1}} + \underbrace{\frac{2}{2^{l-1}} L^l_s \|\hat x\|^{2l}_2}_{\text{Part 2}}
	\]
	We know that $G \geq 0$ so Part 1 is always negative assuming $l$ is odd, and Part 2 is always positive. Therefore, it suffices to find the order $l$ such that 
	\begin{equation}\label{eq:l_condition}
		G^l > (1/2^{l-1}) L^l_s \|\hat x\|^{2l}_2
	\end{equation}
	to be able to make the LHS of \eqref{eq:socp_lifted} negative.
	
%
	
	To derive a sufficient condition for \eqref{eq:l_condition}, we first need a lowerbound on $G$, and to do so, we start with the (sparse) RSC assumption:
	
	\[
		f(M^*) \geq f(\hat x \hat x^\top) + \langle \nabla f(\hat x \hat x^\top), M^* - \hat x \hat x^\top \rangle + \frac{\alpha_s}{2} \|M^* - \hat x \hat x^\top \|^2_F 
	\]
	Then, since $f(M^*) \leq f(\hat x \hat x^\top)$ (as $M^*$ is the global optimum), we have:
	\begin{equation}\label{eq:thm_socp_r1_help1}
		\begin{aligned}
			 0 \geq \langle \nabla f(\hat x \hat x^\top), M^* - \hat x \hat x^\top \rangle + \frac{\alpha_s}{2} \|M^* - \hat x \hat x^\top \|^2_F = \langle \nabla f(\hat x \hat x^\top), M^* \rangle + \frac{\alpha_s}{2} \|M^* - \hat x \hat x^\top \|^2_F
		\end{aligned}
	\end{equation}
	where the equality follows from the FOP condition \eqref{eq:focp_unlifted} for $\hat x$. Also, since $M^*$ is assumed to be positive semidefinite, we have
	\begin{equation}\label{eq:thm_socp_r1_help2}
		\langle \nabla f(\hat x \hat x^\top), M^* \rangle \geq \lambda_{\text{min}}(\nabla f(\hat x \hat x^\top)) \tr(M^*).
	\end{equation}
	Combining \eqref{eq:thm_socp_r1_help1} and \eqref{eq:thm_socp_r1_help2}, we have:
	\[
		-G = \lambda_{\text{min}}(\nabla f(\hat x \hat x^\top)) \leq -\frac{\alpha_s}{2 \tr(M^*)} \|M^* - \hat x \hat x^\top \|^2_F,
	\]
	meaning that
	\begin{equation}\label{eq:G_lower}
		G \geq \frac{\alpha_s}{2 \tr(M^*)} \|M^* - \hat x \hat x^\top \|^2_F
	\end{equation}
	Therefore, if
	\[
		\left( \frac{\alpha_s}{2 \tr(M^*)} \|M^* - \hat x \hat x^\top \|^2_F \right)^l > (1/2^{l-1}) L^l_s \|\hat x\|^{2l}_2,
	\]
	we can conclude that \eqref{eq:l_condition} holds, which implies that the LHS of \eqref{eq:socp_lifted} is negative, directly proving that $\hat x^{\otimes l}$ is not a SOP anymore. Elementary manipulations of the above equation give that a sufficient condition is 
	\begin{equation}\label{eq:l_condition2}
		\|M^* - \hat x \hat x^\top \|^2_F > 2^{1/l} \frac{L_s}{\alpha_s} \|\hat x\|^2_2 \tr(M^*)
	\end{equation}
	We now consider \eqref{eq:thm_socp_distance_lower}, which means that
	\begin{equation}\label{eq:xhat_upper_D}
		\|\hat x\|^2  \leq \frac{\alpha_s}{L_s \tr(M^*)} \|M^* - \hat x \hat x^\top \|^2_F
	\end{equation}
	Subsequently, define a constant $\gamma$ such that:
	\[
		L_s \|\hat x\|^2_2 = \gamma (\frac{\alpha_s}{2 \tr(M^*)} \|M^* - \hat x \hat x^\top \|^2_F) 
	\]
	Then according to Lemma \ref{lem:G_upper} and \eqref{eq:G_lower}, we can conclude that $\gamma \geq 1$. Moreover, \eqref{eq:xhat_upper_D} also means that $\gamma < 2$. So with this new definition, the sufficient condition \eqref{eq:l_condition2} becomes
	\begin{equation}\label{eq:l_condition3}
		1 > \frac{\gamma}{2^{(l-1)/l}}
	\end{equation}
	Since we already know that $1 \leq \gamma < 2$, there always exists a large enough $l$ such that \eqref{eq:l_condition3} holds, which in turn implies that LHS of \eqref{eq:socp_lifted} is negative, proving that $\hat x^{\otimes l}$ is a saddle point with the escape direction $u^{\otimes l}$, proving the claim.
	
	Next, we aim to study how large $l$ needs to be in order for \eqref{eq:l_condition3} to hold. Now, by utilizing Lemma \ref{lem:xhat_upper} again, we know that
	\[
		\gamma = \frac{2 L_s \tr(M^*) \|\hat x\|^2_2}{\alpha_s \|M^* - \hat x \hat x^\top \|^2_F} \coloneqq 2 \beta
	\]
	and we know $\beta \leq 1$ due to assumption \eqref{eq:thm_socp_distance_lower}. So for \eqref{eq:l_condition3} to hold true, we need
	\[
		2^{(l-1)/l} > 2 \beta \implies \frac{l-1}{l} > \log_2(2 \beta) \implies l > \frac{1}{1-\log_2(2 \beta)} 
	\]
\end{proof}

\begin{proof}[Proof of Theorem \ref{thm:socp_corollary}]
	First, consider the following technical lemma, which is proved below this proof,
	\begin{lemma}\label{lem:xhat_upper}
		Given a FOP $\hat x$ of \eqref{eq:unlifted_main}, it holds that
		\begin{equation}\label{eq:xhat_upper}
			\|\hat x\|^2 < \sqrt{\frac{2L_s}{\alpha_s}} \|M^*\|_F
		\end{equation}
\end{lemma}
	Via the above lemma, we know that a sufficient condition to \eqref{eq:thm_socp_distance_lower} is 
	\[
		\|M^* - \hat x \hat x^\top \|^2_F \geq \sqrt{\frac{2L^3_s}{\alpha^3_s}} \|M^*\|_F \tr(M^*)
	\]
	Making the RHS of the above inequality to be smaller than RHS of \eqref{eq:thm_local_region_range} proves the theorem, especially by acknowledging that $M^*$ is rank-1. 
\end{proof}

\begin{proof}[Proof of Lemma \ref{lem:xhat_upper}]
	Lemma 6 of \cite{zhang2021general} states that given an arbitrary constant $\lambda$ and vector $u \in \RR^n$,
	\[
		\|u\|^4_2 \geq \max\{\frac{2L_s}{\alpha_s} \|M^*\|^2_F, (\frac{2\lambda \sqrt{r}}{\alpha_s})^{4/3} \} \implies \|\nabla h(u)\|_F \geq \lambda
	\]
	A simple negation to both sides gives
	\[
		\|\nabla h(u)\|_F < \lambda \implies \|u\|^4_2 < \max\{\frac{2L_s}{\alpha_s} \|M^*\|^2_F, (\frac{2\lambda \sqrt{r}}{\alpha_s})^{4/3}\}
	\]
	If we set $u = \hat x$, then LHS of the above relationship is automatically satisfied for arbitrarily small $\lambda$ since $\|\nabla h(\hat x)\|_F = 0$, and thus we conclude that
	\[
		\|u\|^4_2 < \frac{2L_s}{\alpha_s} \|M^*\|^2_F
	\]
	since $(\frac{2\lambda \sqrt{r}}{\alpha_s})^{4/3}$ can be made arbitrarily small.
\end{proof}

\begin{proof}[Proof of Theorem \ref{thm:socp_z}]
	Again utilizing the assumption that $A_s$ matrices are symmetric( $A_s$ can be converted to be symmetric without altering the observation $b$), we arrive at
	\begin{equation}
	\begin{aligned}
		\text{LHS of \eqref{eq:socp_unlifted}} = 2 \sum_{a,i,j,s}^{[m] \times [n] \circ 4} (A_a)_{sk} (A_a)_{ij} (\hat x_i \hat x_j -z_i z_j) u_k u_s + 
		4 \sum_{a,i,j,s}^{[m] \times [n] \circ 4} (A_a)_{sk} (A_a)_{ij} \hat x_i \hat x_k u_j u_s \geq 0
	\end{aligned}
	\end{equation}
	for any SOP $\hat x$. If we substitute $z$ into the above equation, we obtain that for any $u \in \RR^n$
	\begin{equation}\label{eq:unlifted_z_socp}
		\sum_{a,i,j,s}^{[m] \times [n] \circ 4} (A_a)_{sk} (A_a)_{ij} z_i z_k u_j u_s \geq 0
	\end{equation}
	Then given any $\Delta \in \RR^{n \otimes l}$, we CP decompose ( CANDECOMP, a standard tensor decomposition scheme) it as:
	\[
		\Delta = \sum_{p=1}^R \delta^{p,1} \otimes \dots \otimes \delta^{p,l}
	\]
	where $R$ is the rank of $\Delta$, a finite number. Next,  we consider \eqref{eq:socp_lifted} evaluated at $z^{\otimes l}$, and we have that LHS of \eqref{eq:socp_lifted} equals:
	\begin{equation}
		\begin{aligned}
		  &\sum_{p=1}^R \left[ 2 \sum_{\{a_\alpha, i_\alpha, j_\alpha, s_\alpha, k_\alpha \}_{\alpha=1}^l}^{([m] \circ l) \times \left[ ([n] \circ l) \circ 4 \right]} \left( \prod_{\alpha=0}^l A_{a_\alpha s_\alpha k_\alpha} A_{a_\alpha i_\alpha j_\alpha} z_{i_\alpha} z_{j_\alpha} \delta^{p,\alpha}_{s_\alpha} \delta^{p,\alpha}_{k_\alpha}\right) - \left( \prod_{\alpha=0}^l A_{a_\alpha s_\alpha k_\alpha} A_{a_\alpha i_\alpha j_\alpha} z_{i_\alpha} z_{j_\alpha} \delta^{p,\alpha}_{s_\alpha} \delta^{p,\alpha}_{k_\alpha} \right) \right] + \\
			&4 \sum_{\{a_\alpha, i_\alpha, j_\alpha, s_\alpha, k_\alpha \}_{\alpha=1}^l}^{([m] \circ l) \times \left[ ([n] \circ l) \circ 4 \right]} \prod_{\alpha=0}^l A_{a_\alpha s_\alpha k_\alpha} A_{a_\alpha i_\alpha j_\alpha} z_{i_\alpha} z_{k_\alpha} \delta^{p,\alpha}_{j_\alpha} \delta^{p,\alpha}_{s_\alpha} \\
		&= \sum_{p=1}^R  4 \sum_{\{a_\alpha, i_\alpha, j_\alpha, s_\alpha, k_\alpha \}_{\alpha=1}^l}^{([m] \circ l) \times \left[ ([n] \circ l) \circ 4 \right]} \prod_{\alpha=0}^l A_{a_\alpha s_\alpha k_\alpha} A_{a_\alpha i_\alpha j_\alpha} z_{i_\alpha} z_{k_\alpha} \delta^{p,\alpha}_{j_\alpha} \delta^{p,\alpha}_{s_\alpha} \\
		&= 4 \sum_{p=1}^R  \prod_{\alpha=0}^l  \left( \sum_{a,i,j,s}^{[m] \times [n] \circ 4} A_{a_\alpha s_\alpha k_\alpha} A_{a_\alpha i_\alpha j_\alpha} z_{i_\alpha} z_{k_\alpha} \delta^{p,\alpha}_{j_\alpha} \delta^{p,\alpha}_{s_\alpha} \right) \geq 0
		\end{aligned}
	\end{equation}
	where the last inequality follows from \eqref{eq:unlifted_z_socp}. 
	
	\end{proof}

\end{document}